\documentclass{article}

\usepackage{arxiv}
\usepackage{amsmath}
\usepackage[utf8]{inputenc} 
\usepackage[T1]{fontenc}    
\usepackage{hyperref}       
\usepackage{url}            
\usepackage{booktabs}       
\usepackage{amsfonts}       
\usepackage{nicefrac}       
\usepackage{microtype}      
\usepackage{cleveref}       
\usepackage{lipsum}         
\usepackage{graphicx}
\usepackage{amsthm}
\newtheorem{thm}{Theorem}

\title{Asymptotics of Some Feynman-Kac  Functionals}

\newif\ifuniqueAffiliation
\uniqueAffiliationtrue

\ifuniqueAffiliation 
\author{ Charles Hagwood \\
	National Institute of Standards and Technology (Retired)\\
	100 Bureau Drive Stop 881\\
	Gaithersgurg, MD 20899 \\
	\texttt{rchagw@gmail.com} }
	\date{} 
\begin{document}
\maketitle

\begin{abstract}
Methods were  initiated by  Mark Kac and Richard Feynman to evaluate 
 random functionals of the form  $\int^t_0V(X_s)ds$ for  a nonnegative $V$ and  a Markov process $X_t$. Their results evolved into the well known Feynman Kac formula.   Functionals of this type appear in both theoretical and applied applications in partial differential equations,  quantum physics, mathematical finance, control theory, etc. Here the time average of one such  functional associated with  the Feynman Kac formula  is studied. In real time applications where only the path is observed, the  time average often is a better predictor than the functional  at its last observation point..  It represents quantities, e.g., the long term average cost or wealth, the long term average velocity.  As a statistic, it is of interest to determine if it has an asymptotic limit and to determine that limit.  An expression is derived for its asymptotic  time average.
\end{abstract}


\section{Introduction}
For $X_t, t\geq 0$ a Markov process on a probability space  $(\Omega,{\cal F},P)$ with state space $I\subset \mathbb{R}$, consider  random integrals of the form
\begin{equation}\label{exponentialIngegral}
e_K(t)=\int^T_tq(X_s)e^{-\int^s_tK(X_\tau)d\tau}ds
\end{equation}
for nonzero positive real valued functions  $K$  and $q$.   These functionals appear prominently in the Feynman-Kac formula.   Kac called them Wiener functionals \cite{kac1949distributions} and he established a relationship between them and the Schrodinger equation \cite{kesten1986influence}.  This was similarly done by Feynman using path integrals \cite{feynman1948space}.   Their joint  work  is now  commonly referred to as the Feynman-Kac formula. Treatments of the Feynman-Kac formula can be found in the textbooks \cite{oksendal2013stochastic}, \cite{karatzas2012brownian}, \cite{bhattacharya2009stochastic}, \cite{rogers2000diffusions}, \cite{freidlin1985functional}. These functionals  are extensively investigated in a more general form $e_V(T)=\int^T_0V(X_t)dt$  for  functionals $V$ and $T$ fixed or random.  Moments, large deviations and connections to Schrodinger's equation  of $e_V(t)$ can be found in \cite{fitzsimmons1999kac}, \cite{darling1957occupation},\cite{cameron1945evaluation},  \cite{kac1951some}, \cite{takeda2011large}, \cite{ioannis1980stochastic}, \cite{takeda2008large}, \cite{chung1980stopped}  and \cite{musiela1986kac}.

The contribution of this paper is an asymptotic  expression  for $\bar{e}_T$ the time average  process of $e_K(t)$.  
 Subject to some restrictions, when $X_t$ has an invariant measure $m(x), x\in I$ it is shown 
\begin{equation}\label{result}
\bar{e}_T=\frac{1}{T}\int^T_0e_K(t)dt \rightarrow \int_I \frac{q(x)}{K(x)}m(x)dx\quad (a.e.).
\end{equation}
$0<T\rightarrow \infty$.   The intergrand $q(x)/K(x)$ in \eqref{result} acts like an action/reaction ratio.   For example, in the  parabolic partial differential equation associated with the Feynman-Kac formula $K(x)$ represents a  potential (barrier) as in the Schrodinger equation and $q(x)$ is associated with a source(sink) or external forcing term acting on the particle.  

 The  functional $e_K(t)$  occurs naturally in stochastic optimal control problems such as the model
\begin{align}
dX_t &=(\mu(t,u_t)X_t +q(t,u_t))dt +\sigma(t,u_t)dB_t \quad 0<t<T\label{BlackSholes}\\
X_T&=x
\end{align}
for $B_t$  a Brownian motion and $u_t$  a control satisfying a  stochastic differential equation 
\begin{align}\label{controlVar}
du_t &=a(t,u_t)dt+b(t,u_t)d\tilde{B}_t\quad 0<t<T\\
u_T&=v
\end{align}
where $\tilde{B}_t$ represents  a second Brownian motion, see  \cite{borkar2005controlled},  \cite{stoikov2005dynamic}.  Its solution is 
\begin{equation}
X_t=  xe^{\int^T_t\mu(\tau,u_\tau)d\tau}-\int^T_tq(s,u_s)e^{\int^s_t\mu(\tau,u_\tau)d\tau}ds -\int^T_t\sigma(s,u_s)e^{\int^s_t\mu(\tau,u_\tau)d\tau}dB_s.
\end{equation}
In  mathematical finance the consumption investment problem is of this form, see \cite{merton1971optimization}, \cite{karatzas1989optimization}.    Another application involves nanorod measurement technology  \cite{mulholland2021effect}.  This example involves measuring a nanorod or a fiber using a procedure analogous to the Millikan oil drop experiment \cite{millikan1910isolation}  that  determined the  elementary electrical  charge $e$  of the electron. Examples of the procedure can be found in \cite{larriba2021size},\cite{mulholland2021effect},\cite{li2012effect} \cite{hagwood2019limiting}.

\section{Results}
The proof of \eqref{result} relies on adapting the  Cauchy Mean Value Theorem  \cite{bartle1964elements} and its proof to this problem. We make it accessible in order to easily follow the proof of Theorem 1.    It says, for $f(x)$ and $g(x)$ two functions continuous on $[a,b]$ and differentiable on $(a,b)$,  there exists a constant  $a<c<b$  such that  for $g(a)\ne g(b)$
\begin{equation}
\frac{f(b)-f(a)}{g(b)-g(a)} =\frac{f'(c)}{g'(c)}.
\end{equation}
 The proof is  simple.  Let
\begin{equation}\label{CMVFunction}
h(x)= f(x)-\frac{f(b)-f(a)}{g(b)-g(a)}g(x).
\end{equation}
 Since $h(b)=h(a)$,  by Rolle's Theorem there exists $a<c<b$  such that $h'(c)=0$, i.e.
\begin{equation}\label{h_fcn_CMV}
h'(c)=f'(c)-\frac{f(b)-f(a)}{g(b)-g(a)}g'(c)=0
\end{equation}

\begin{thm}\label{Thm1} Let  $X_t, t\geq 0$ be a process with continuous sample paths. The following result holds  for nonzero  continuous functions $K(x)$ and  $q(x), x\in \mathbb{R}$ and for $0\leq t\leq T$
\begin{equation}\label{thm1i}
 \int^T_tq(X_s)e^{-\int^s_tK(X_\tau)d\tau}ds=\frac{q(X_{\xi_{tT}})}{K(X_{\xi_{tT}})}(1-e^{-\int^T_tK(X_\tau)d\tau})
\end{equation}
where
\begin{equation}\label{changeOfTimeVar}
\xi_{tT}=\inf\{s: t<s<T, \frac{K(X_s)}{q(X_s)}=r_T(t)\}
\end{equation}
and where $r_T(t)$  is given by
\begin{equation}\label{bdyCurve}
r_T(t)=\frac{e^{-\int^t_0K(X_\tau)d\tau}-e^{-\int^T_0K(X_\tau)d\tau}}{\int^T_tq(X_s)e^{-\int^s_0K(X_\tau)d\tau}}=\frac{1-e^{-\int^T_tK(X_\tau)d\tau}}{\int^T_tq(X_s)e^{-\int^s_tK(X_\tau)d\tau}}.
\end{equation}

\end{thm}

\begin{proof}:
\noindent One may write
\begin{equation}
\int^T_tq(X_s)e^{-\int^s_tK(X_\tau)d\tau}ds=\frac{\int^T_tq(X_s)e^{-\int^s_0K(X_\tau)d\tau}ds}{e^{-\int^t_0K(X_\tau)d\tau}}=\frac{f(t)}{g(t)}
\end{equation}
where the functions $f,g$ are given by
\begin{equation}
f(t) =\int^T_tq(X_s)e^{-\int^s_0K(X_\tau)d\tau}ds\qquad g(t) =e^{-\int^t_0K(X_\tau)d\tau}.
\end{equation}
For these two  functions, applying the Cauchy Mean Value Theorem,  there exists $\xi_{tT}$ such that
\begin{equation}
\frac{f(T)-f(t)}{g(T)-g(t)}=\frac{f'(\xi_{tT})}{g'(\xi_{tT})} \quad t <\xi_{tT}< T
\end{equation}
or equivalently, since $f(T)=0$
\begin{equation}
\frac{f(t)}{g(t)}=\frac{f'(\xi_{tT})}{g'(\xi_{tT})}(1-\frac{g(T)}{g(t)})
=\frac{q(X_{\xi_{tT}})}{K(X_{\xi_{tT}})}(1-\frac{g(T)}{g(t)})
\end{equation}
which is the postulated result in \eqref{thm1i}.

The time change $\xi_{tT}$ may be determined by mimicking the proof of the Cauchy Mean Value Theorem. There may be more than one solution to \eqref{h_fcn_CMV}. The   smallest is taken to define  $\xi_{tT}$ 
\begin{equation}
\xi_{tT} =\inf\{s: t<s<T, h'(s)=0\}
\end{equation}
where
\begin{equation}
h(x)= f(x)-\frac{f(T)-f(t)}{g(T)-g(t)}g(x)=f(x)+\frac{f(t)}{g(T)-g(t)}g(x).
\end{equation}
 The derivative of  $h(x)$ is
\begin{equation}
h'(x) =f'(x)+\frac{f(t)}{g(T)-g(t)}g'(x)=-q(X_x)e^{-\int^x_0K(X_\tau) d\tau}-\frac{f(t)}{g(T)-g(t)}K(X_x)e^{-\int^x_0K(X_\tau) d\tau}
\end{equation}
The equality $h'(x)=0$ reduces to
\begin{equation}
-q(X_x)-\frac{f(t)}{g(T)-g(t)}K(X_x)=0
\end{equation}
or
\begin{equation}
\frac{ K(X_x)}{q(X_x)}=\frac{g(t)-g(T)}{f(t)}
\end{equation}
Thus
\begin{align}\label{chTime1}
\xi_{tT}&=\inf\{x: t<x<T,\frac{ K(X_x)}{q(X_x)}=\frac{g(t)-g(T)}{f(t)}\}
\end{align}
as was postulated in \eqref{bdyCurve}.
\end{proof}

\begin{thm} \label{Thm2} Let $X_t$ be a process having continuous sample paths and in addition assume  $X_t$ has an invariant measure $m(x), x\in I\subset \mathbb{R}$, $|q(x)|\leq L$, $K(x)$ is a positive continuous lower bounded  function with $K(x) >M,$ for  $x\in \mathbb{R}$, $M$ a positive constant and the derivative of $q(x)/K(x)$ is bounded, then 
\begin{equation}\label{limitThm}
\lim_{T\rightarrow \infty}\frac{1}{T}\int^T_0\int^T_tq(X_s)e^{-\int^s_tK(X_\tau)d\tau}dsdt=\int_I\frac{q(x)}{K(x)}m(dx).
\end{equation} 
\end{thm}
\begin{proof} Fix $0<\epsilon <1$ and partition the  integral $I_T$ on the lefthand side of \eqref{limitThm} as
\begin{equation}
I_T=\frac{1}{T}\int^T_{T\epsilon}\int^T_tq(X_s)e^{-\int^s_tK(X_\tau)d\tau}dsdt +\frac{1}{T} \int^{T\epsilon}_0\int^{T}_tq(X_s)e^{-\int^s_tK(X_\tau)d\tau}dsdt =I_1^\epsilon+I_2^\epsilon.
\end{equation}
Applying the assumptions
\begin{align}
|I_2^\epsilon|&=|\frac{1}{T}\int^{T\epsilon}_0\int^{T}_tq(X_s)e^{-\int^s_tK(X_\tau)d\tau}dsdt|<\frac{L}{T}\int^{T\epsilon}_0\int^{T}_te^{-(s-t)M}dsdt\\
&=L(\frac{e^{-MT}-e^{-T(1-\epsilon)M}}{M^2T}+\frac{\epsilon}{M})\label{firstepsilson}
\end{align}
which goes to zero as $T\rightarrow \infty$ and then $\epsilon \rightarrow 0$.  

Write  $I_1^\epsilon$ as follows using  \eqref{thm1i} from Theorem \ref{Thm1}
\begin{equation}\label{I1Partition}
I_1^\epsilon =\frac{1}{T}\int^T_{T\epsilon}\frac{q(X_{\xi_t})}{K(X_{\xi_t})}dt-\frac{1}{T}\int^T_{T\epsilon}\frac{q(X_{\xi_t})}{K(X_{\xi_t})}e^{-\int^T_tK(X_\tau) d\tau}dt=I^\epsilon_{11}(T)-I^\epsilon_{12}(T)
\end{equation}
Using assumptions   
\begin{align}
|I^\epsilon_{12}(T)| &=|\frac{1}{T}\int^T_{T\epsilon}\frac{q(X_{\xi_t})}{K(X_{\xi_t})}e^{-\int^T_tK(X_\tau) d\tau}dt |\leq \frac{1}{T} \frac{L}{M}\int^T_{T\epsilon}e^{-(T-t)M}dt\\
&=(\frac{L}{T})\frac{-1+e^{-T(1-\epsilon)M}}{M^2}                   \rightarrow 0
\end{align}
as $T\rightarrow \infty$ for all $0<\epsilon <1$.

Since $X_t$ has invariant measure $m(x)$
\begin{equation}\label{lastIntegral}
\frac{1}{T}\int^T_{T\epsilon}\frac{q(X_t)}{K(X_t)}dt\rightarrow (1-\epsilon)\int_I\frac{q(x)}{K(x)}m(x)dx
\end{equation}
as $T\rightarrow \infty$.  For the time changed process
\begin{equation}
I^\epsilon_{11}(T)=\frac{1}{T}\int^T_{T\epsilon}\frac{q(X_{\xi_{tT}})   }{K(X_{\xi_{tT}})}dt = \frac{1}{T}\int^T_{T\epsilon}\frac{q(X_t)}{K(X_t)}dt  + \frac{1}{T}\int^T_{T\epsilon}\Delta_{tT}dt
\end{equation}
where
\begin{equation}
\Delta_{tT}=\frac{q(X_{\xi_{tT}})   }{K(X_{\xi_{tT}})}-\frac{q(X_t)}{K(X_t)}\quad T\epsilon \leq t \leq  T.
\end{equation}
By a   first order Taylor series expansion
\begin{equation}
\Delta_{tT}=F(\tilde{X}_t)(X_{\xi_{tT}}-X_t)
\end{equation}
where $F(x)$ is the bounded  derivative of $q(x)/K(x)$ and $X_{\xi_{tT}}<\tilde{X}_t<X_t$. 
Below it is shown $\xi_{tT}\rightarrow   T$ as $t\rightarrow T$. Recall
\begin{equation}
\xi_{tT}=\inf\{s:t<s<T, \frac{K(X(s))}{q(X_s)}=r_T(t)\}
\end{equation}
where 
\begin{equation}
r_T(t)=\frac{e^{-\int_0^tK(X_\tau)d\tau}-e^{-\int_0^TK(X_\tau)}d\tau}{\int^T_tq(X_s)e^{-\int_0^sK(X_\tau)d\tau}ds}
\end{equation}
By L'Hospital's Rule, $\lim_{t\rightarrow T}r_T(t)=K(X_T)/q(X_T)$, implying $\xi_{tT}\rightarrow   T$ as $t\rightarrow T$. Therefore letting $t\rightarrow T$ and then $T\rightarrow \infty$ one gets using continuity of sample paths
\begin{equation}
X_{\xi_{tT}}-X_t=X_{\xi_{tT}}-X_T-(X_t-X_T) \rightarrow 0.
\end{equation}
Thus, using the boundedness of $F$ and $T\epsilon\leq t \leq T,$ then as $T\rightarrow \infty$,  $t\rightarrow \infty$ and it follows
\begin{equation}
\Delta_{tT}=F(\tilde{X}_t)(X_{\xi_{tT}}-X_t)\rightarrow 0
\end{equation}
which implies
\begin{equation}
 \frac{1}{T}\int^T_{T\epsilon}\Delta_{tT}dt \rightarrow 0.
\end{equation}
 Finally let $\epsilon \rightarrow 0$ in \eqref{firstepsilson} and \eqref{lastIntegral} to get the desired result in \eqref{limitThm}.
\end{proof}

\bibliographystyle{unsrt}

\end{document}






\end{document}